\documentclass[11pt]{amsart}
\usepackage[utf8]{inputenc}
\usepackage{amsthm}
\usepackage{amsfonts}
\usepackage{amssymb}
\usepackage{mathrsfs}
\usepackage{pdfpages}
\usepackage{comment}
\usepackage{amsmath}
\usepackage{graphicx}
\usepackage[all]{xy}
\usepackage{endnotes}
\usepackage{enumerate} 
\usepackage[hidelinks]{hyperref}
\usepackage{fullpage}

 \newtheorem{thm}{{Theorem}}[section]
 
 \newtheorem{prop}[thm]{Proposition}
 
 \newtheorem{cor}[thm]{{Corollary}}
 \newtheorem{defn}[thm]{{Definition}}
 \newtheorem{ejemplo}[thm]{{Example}}

\newcommand{\N}{{\mathbb N}}

\newcommand{\R}{{\mathbb R}}

\newcommand{\dom}{\textit{\sf dom}}
\newcommand{\im}{\textit{\sf im}}
\newcommand{\ev}{\text{\sf ev}}

\newcommand{\cantor}{2^{\N}}
\newcommand{\ideal}{\mathcal{I}}

\def\sd{{\scriptstyle \triangle}}

\title{Pettis property for Polish inverse semigroups}

\author{Karen Arana}
\address{Escuela de Matem\'aticas, Universidad Industrial de Santander, C.P. 680001, Bucaramanga - Colombia.}
\email{daniarana22@gmail.com}

\author{Jerson P\'erez}
\address{Escuela de Matem\'aticas, Universidad Industrial de Santander, C.P. 680001, Bucaramanga - Colombia.}

\email{jersonenrique\_64@hotmail.com}

\author{Carlos Uzc\'{a}tegui}
\address{Escuela de Matem\'aticas, Universidad Industrial de Santander, C.P. 680001, Bucaramanga - Colombia.}
\email{cuzcatea@saber.uis.edu.co}


\date{}

\begin{document}

\maketitle
\begin{abstract}
We study a property about Polish inverse semigroups similar to the classical theorem of Pettis about Polish groups. In contrast to what happens with Polish groups, not every Polish inverse semigroup have the Pettis property.  We present several examples of Polish inverse subsemigroup of the symmetric inverse semigroup $I(\N)$ of all partial bijections between subsets of $\N$. We also study whether our examples  satisfy automatic continuity.
\end{abstract}

\section{Introduction}
Let $\mathcal{C}$ be a class of topological groups, typically, $\mathcal{C}$ consists of all  Polish groups or more generally all second countable topological groups. A topological group $G$ has automatic continuity  with respect to the class $\mathcal{C}$ if every  homomorphism from $G$ into a group in $\mathcal{C}$ is continuous. There has been some increasing interest in the phenomena of automatic continuity in Polish groups \cite{KechrisRosenthal2007,Rosendal2009,RosendalSolecki2007,Sabok2019}.
A typical argument to get automatic continuity for Polish groups uses the  classical Pettis theorem: For every non meager Baire measurable subset $A$ of a Polish group $G$ there is an open set $V$ such that $1_G\in V\subseteq A^{-1}A$ (see, for instance, \cite[Theorem 9.9]{kechris1995}).  A well known consequence  of this is that every Baire measurable homomorphism between Polish groups is automatically continuous (see for instances \cite[Theorem 9.10]{kechris1995}). 
The restriction to Baire measurable homomorphism is crucial, for instance, it is well known that the additive groups $\R$ and $\R^2$ are isomorphic as vector spaces over the rationals but the isomorphism cannot  be continuous. To get the full automatic continuity the group has to satisfy some special propery (like having ample generics \cite{KechrisRosenthal2007} or the Steinhauss property \cite{RosendalSolecki2007}). Automatic continuity of a Polish group is closely related to the fact that the group carries a unique Polish group topology; the interested reader can consult \cite{Rosendal2009} for more information about it. 

Recently, there have been some  works about automatic continuity for Polish semigroup topologies (see \cite{elliott2020,elliott2022,MES, PerezUzca2022}).  In this paper we explore a possible generalization of Pettis theorem now for Polish inverse semigroups. 
We say that a topological inverse semigroup $S$ has the {\em Pettis property} if for every non meager  $A\subseteq S$ with the Baire property, there is an idempotent $e\in S$ and an open set $V$ such that $e\in V$ and $V\subseteq A^{-1}A$. 
As we said above, every Polish group has the Pettis property as an inverse semigroup. In contrast, we show that no every Polish inverse semigroup has the Pettis property. Another difference between these two algebraic structures is that the Pettis property for Polish inverse semigroups does not imply that  Baire measurable homomorphism are automatically continuous. One the purposes of this article is to provide examples of Polish inverse semigroups realizing each  Boolean combinations of these properties. 

The symmetric inverse  semigroup  $I(X)$ on an arbitrary non empty set $X$ consists of all partial bijections between subsets of $X$. It is fundamental in semigroup theory since, by  a well known result of Wagner and Preston, every inverse semigroup is isomorphic to a subsemigroup of $I(X)$ for some $X$ (see, for instance, \cite{Howie}).  $I(X)$ carries a very natural inverse semigroup topology introduced recently in \cite{elliott2020, PerezUzca2022} which generalizes the product topology on $X^X$ (with $X$ discrete). In the particular case of $X$ being $\N$,  $I(\N)$ turns out to be a Polish (i.e. completely metrizable and separable) inverse semigroup containing  the symetric group $S_\infty(\N)$ (with its usual product topology) as a Polish subgroup.  
In \cite{elliott2020} is shown that $I(\N)$ has automatic continuity with respect to all second countable inverse semigroups.  We show that $I(\N)$ does not have the Pettis property. 
We study several Polish inverse subsemigroups of $I(\N)$. In particular, we analyze  inverse subsemigroups generated by families of groups of the form $S_\infty(B)$ for $B\subseteq \N$.

\section{Preliminaries}

A topological space is Polish if it is completely metrizable and separable. We refer to \cite{kechris1995} as a general reference for the descriptive set theory of Polish spaces.   
$\N$ denotes the non negative integer. As usual, we identify a subset $A\subseteq \N$ with its characteristic function and thus a collection $\mathcal{C}$ of subsets of $\N$ is seen as a subset of $\cantor$ and we can talk about closed, open, $F_\sigma$, $G_\delta$ etc. collections of subsets of $\N$. 
A family of subsets of a set $X$ is {\em almost disjoint} if $A\cap B$ is finite for every pair of sets in the family.  A subset $A$ of a topological space  has the {\em Baire property} or is {\em Baire measurable}, if there is an open set $V$ such that $A\sd V$ is meager.

A {\em semigroup} is a non-empty set $S$ together with an associative binary operation $\circ$. To simplify the notation we sometimes write $st$ in placed of $s\circ t$.
A semigroup  $S$ is  {\em regular},  if for all $s\in S$, there is $t\in S$ such that $st s=s$ and $tst=t$. In this case, $t$ is called  an inverse of $s$. If each element have a unique inverse,  $S$ is an {\em inverse} semigroup and, in this case, $s^\ast$ denotes  the inverse of $s$. An element $s$ of a semigroup is {\em idempotent} if $ss=s$. We denote by $E(S)$ the collection of idempotents of $S$. We use Howie's book \cite{Howie} as a general reference for semigroup theory

Let $\tau$ be a topology  on a semigroup $S$.  If the multiplication $S\times S\to S$ is continuous, we call $S$ a {\em topological semigroup}.  An inverse semigroup $S$ is called topological, if it is a topological semigroup and the function $i:S\to S$, $s\to s^\ast$ is continuous.
We refer the reader to \cite{Carruthetal} as a general reference for topological semigroups.

The symmetric inverse semigroup on a set $X$ is defined as follows:
$$
I(X) = \{f : A \rightarrow B |\;A, B \subseteq X\; \mbox{and $f$ is bijective}\}.
$$
For $f : A \rightarrow B$ in $I(X)$ we denote $A = \dom(f)$  and $B = \im(f)$. The operation on $I(X)$ is the usual composition, namely, given $f, g \in I(X)$, then $f \circ g$ is defined by letting $\dom(f \circ g)=g^{-1}(\dom(f) \cap \im(g))$ and $(f \circ g)(x) = f(g(x))$, if $x \in \dom(f \circ g)$. The idempotents  of $I(X)$ are the partial identities $1_{A} : A \rightarrow A, 1_{A}(x) = x$ for all $x \in A$ and $A \subseteq X$.
Notice that $1_{\emptyset}$ is the empty function which also belongs to $I(X)$.
Let $S_{\infty}(X)$ be the symmetric group, that is, the collection of all bijection from $X$ to $X$. 

The following functions play an analogous role in $I(X)$ as the projection functions do in  the product space $X^X$. Let $D_{x} = \{f\in I(X):\; x\in 
\dom(f)\}$ and $2^{X}$ denotes the power set of $X$.
\[
\begin{array}{lll}
\dom : I(X) \rightarrow 2^{X}, &f \mapsto  \dom(f),\\
\im : I(X) \rightarrow 2^{X}, & f \mapsto  \im(f),\\
\ev_{x} : D_x \rightarrow X, & f \mapsto  f(x),\;\mbox{for $x \in X$}.
\end{array}
\]
For $x, y \in X$, let
\[
\begin{array}{lcl}
v(x, y) & = &  \{f \in I(X)\; |\; x \in \dom(f) \;\mbox{and}\; f(x) = y\},\\
w_{1}(x) & = & \{f \in I(X)\; |\; x \notin \dom(f)\},\\
w_{2}(y) & = & \{f \in I(X)\; |\; y \notin \im(f)\}.
\end{array}
\]

The sets $v(x, y)$ are motivated by the usual subbase for the product topology on $X^{X}$. The topology generated by  the sets $v(x,y)$, $w_1(x)$ and $w_2(x)$ is denoted by $\tau_{pp}$. This topology was defined in \cite{elliott2020,PerezUzca2022}. The topology $\tau_{pp}$ is the smallest Hausdorff topology that makes continuous the functions $\dom$, $\im$ and $\ev_x$ (for $x\in X$), where $2^X$ is given the product topology.  For this paper, we work only with $X=\N$ and, unless said otherwise,  we always use this topology on $I(\N)$. $(I(\N), \tau_{pp})$ is a Polish inverse semigroup. Convergence in $I(\N)$ is as follows (see \cite{PerezUzca2022}).

\begin{prop}
\label{conv}
Let $(f_n)_n$  be a sequence in $I(\N)$ and $f\in I(\N)$. Then, $f_n\to f$ if and only if  the following conditions hold.

\begin{itemize}
\item[(i)] For all $x\in \dom(f)$ there is $n_{0}\in \N$ such that  $x\in \dom(f_n)$  and $f_n(x)=f(x)$ for all  $n\geq n_0$.

\item[(ii)]  For all $x\not\in \dom(f)$ there is $n_0\in\N$ such that $x\not\in \dom(f_n)$  for all $n\geq n_0$.
\end{itemize}
\end{prop}

\section{The Pettis property for inverse semigroups}

A topological inverse semigroup $S$ has the {\em Pettis property} if for every non meager $A\subseteq S$ with the Baire property, there is an idempotent $e\in S$ and an open set $V$ such that $e\in V$ and $V\subseteq A^{-1}A$. 

Notice that if $x$ is an isolated point of $S$ and $xx^{-1}$ is not isolated, then $S$ does not have the Pettis property by a trivial reason: take $A=\{x\}$, then $A$ is not meager but does not satisfy the conclusion of the Pettis property.  In contrast to what happens for topological groups,  a non discrete topological semigroup can have an isolated idempotent. The following proposition is a simple criterion for having the Pettis property. It will be used in the sequel.

\begin{prop}
\label{criterio-1}
Let $S$ be a Polish inverse semigroup. Let $I$ be the collection of isolated points of $S$. Suppose that $x^{-1}x\in I$ for all $x\in I$. If there is a countable collection $\{O_n: \; n\in \N\}$ of open subgroups of $S$ such that $I\cup \bigcup_n O_n$ is dense in $S$, then $S$ has the Pettis property.

\end{prop}

\begin{proof}
Let $A\subseteq S$ be a non meager set with the Baire property. As $I\cup \bigcup_n O_n$ is open dense, there are two cases:

(a) Suppose $A\cap I\neq \emptyset$ and  let $x\in A\cap I$, then $e=x^{-1}x$ is isolated and $e\in int(A^{-1}A)$.

(b) Suppose $A\cap I=\emptyset$. Then, there is $n$ such that $A\cap O_n$ is not meager. Clearly $A\cap O_n$ has the Baire property. Since $O_n$ is a Polish group, by Pettis theorem, there is a set $V\subseteq O_n$ open in $O_n$ (hence, open in $S$ too) containing the identity $e$ of $O_n$ such that $V\subseteq (A^{-1}A)\cap O_n$. 
\end{proof}

It is a classical result that any non meager Baire measurable subgroup of a Polish group  is clopen (see, for instance, \cite[Exercise 9.11]{kechris1995}). Our next result shows a similar fact for subgroups of  Polish inverse semigroups.

\begin{prop}
Let $S$ be a Polish inverse semigroup with the Pettis property. Let $G$ be a subgroup of $S$ and suppose that  $l_x:S\to S$ given by $l_x(y)=yx$ is an open map for each $x\in G$. If $G$ is non meager and Baire measurable, then $G$ is clopen.
\end{prop}

\begin{proof}
By the Pettis property of $S$, there is $V$ open such that $1_G\in V\subseteq G$. Since $l_x$ is an open map for each $x\in G$ and $G=\bigcup_{x\in G}Vx$,  $G$ is open. Let $H=\overline{G}$. Then $H$ is a Polish subgroup of $S$ and $G$ is an open subgroup of $H$. Then $H=G$, as it happens on topological groups (if $x\in H\setminus G$, then $G\cap Gx=\emptyset$, which contradicts that both are open in $H$ and $G$ is dense in $H$). 
\end{proof}

\section{Inverse subsemigroups of $I(\N)$ generated by groups}

In this section we study subsemigroups of $I(\N)$ generated by families of subgroups of $I(\N)$, that is, by groups of the form $S_\infty(A)$.  More precisely, we analyze the inverse submsemigroup of $I(\N)$ generated by a collection of groups $S_\infty(B_i)$ where $\{B_i: i\in I\}$ is an almost disjoint family. 
We find conditions that guarantee that those semigroups are closed in $I(\N)$ and thus Polish. We also present examples of such semigroups with and without the Pettis property.

For each $k\in \mathbb{N}$ and $A,B\subseteq \mathbb{N}$, we define 
$$
I_k(A,B)=\{f\in I(\mathbb{N})\; |\;\dom(f)\subseteq A\textit{, }\im(f)\subseteq B \textit{ and } |\dom(f)|=k\}.
$$
Notice that $I_0(A,B)=\{1_\emptyset\}$. Also, we let
\[
S(A,B)=\{f\in I(\N):\;\dom(f)=A\;\;\&\;\; \im(f)=B\}.
\]

\begin{prop} 
\label{2.2}
Let  $\{B_i\;|\;i\in \N\}$ be an almost disjoint family of infinite subsets of $\N$. Given   $n\in \N$ we set
$$
S=\displaystyle\bigcup_{i\in \N}S_{\infty}(B_i)\cup \bigcup_{i,j\in \N}\bigcup_{k=0}^{n}I_k(B_i,B_j).
$$ 
Then, $S$ is a subsemigroup of $I(\N)$ iff $|B_i\cap B_j|\leq n$, for each $i\not=j$.
\end{prop}

\begin{proof}
Suppose that $S$ is a semigroup. Let $i\not=j$ and  $f\in S_\infty(B_j)$ and $g\in S_\infty(B_i)$. Then $|\dom(f\circ g)|=|g^{-1}(\im(g)\cap \dom(f))|=|B_i\cap B_j|$. Since $f\circ g\in S$ and $B_j\cap B_i$ is finite,  there  is $0\leq k\leq n$ such that $f\circ g \in I_k(B_i,B_j)$. Thus, $|\dom(f\circ g)|=|B_i\cap B_j|=k\leq n$.

Conversely, suppose that $|B_i\cap B_j|\leq n$ for all $i\not=j$ and let $f,g\in S$. We have to show that $f\circ g\in S$. If $f,g\in S_\infty(B_i)$, for some $i\in \N$, then $f\circ g\in S_\infty(B_i)$. If $f\in S_\infty(B_j)$ and $g\in S_\infty(B_i)$, with $i\not=j$, then $f\circ g\in I_{k}(B_i, B_j)$, where $k=|B_{i}\cap B_{j}| $. Finally, if $g\in I_k(B_i,B_j)$, then $f\circ g\in I_l(B_i,B_p)$ for some $l\leq k$ and $p\in \N$.

\end{proof}

\begin{prop} 
\label{semigrupogenerado}
Let $\{B_i\;|\;i\in \N\}$ be a family of infinite subsets of $\N$ such that, for some fixed $n$, $|B_i\cap B_j|=n$ for all $i\neq j$.  
Let 
$$
S=gen\displaystyle\left(\bigcup_{i\in \N}S_{\infty}(B_i)\right)
$$
be the inverse subsemigroup generated by $\bigcup_{i\in \N}S_{\infty}(B_i)$. Then 
\begin{equation}
\label{generado}
S=\bigcup_{i\in \N}S_{\infty}(B_i)\cup \bigcup_{i,j\in \N}\bigcup_{k=0}^{n}I_k(B_i,B_j).
\end{equation}
\end{prop}

\begin{proof}
First, notice that  $\subseteq$ in (\ref{generado}) follows from Proposition \ref{2.2}.

%Let $f\in S$, then $f$ is of the form $f=f_1\circ\cdots \circ f_m$, where each $f_i\in S_{\infty}(B_{k_i})$ for some $k_i\in I$.  We can assume that  $m\geq2$. Since $\dom(f)\subseteq \dom(f_{m-1}\circ f_m)$ and $|\dom(f_{m-1}\circ f_n)|\leq n$,  we have that $|\dom(f)|\leq n$. On the other hand, $\dom(f)\subseteq \dom(f_m)=B_{k_m}$ and $\im(f)\subseteq \im(f_1)=B_{k_1}$, from which it follows that $f\in I_k(B_{k_n},B_{k_1})$, where $k=|\dom(f)|$. This finished the proof of $\subseteq$ in (\ref{generado}).

To show $\supseteq$ in (\ref{generado}), we first verify that  $I_n(B_i,B_j)\subseteq S$ for all $i,j\in \N$. Suppose $i\not=j$. Let $f\in I_n(B_i,B_j)$, $\dom(f)=\{r_1,\cdots ,r_n\}$ and $B_i\cap B_j=\{s_1,\cdots ,s_n\}$. Pick  $g\in S_\infty(B_j)$ and $h\in S_\infty(B_i)$ such that $g(s_i)=f(r_i)$ and $h(r_i)=s_i$ for each $i\in \{1,\cdots, n\}$. It is easy to see that $f=g\circ h$, thus $f\in S$.

It remains to show that $I_n(B_i, B_i)\subseteq S$ for each $i\in \N$. Let $f\in I_n(B_i, B_i)$ and  fix $j\in \N$ with $i\not=j$. Then, $\dom(f)=\{r_1,\cdots,r_n\}$ and $B_i\cap B_j=\{s_1,\cdots ,s_n\}$. Pick $h\in S_\infty(B_i)$ and $g\in S_\infty(B_j)$ such that $h(r_i)=s_i$ and  $g(s_i)=f(r_i)$ for each $i\in \{1,\cdots, n\}$. Then, $f=g\circ 1_{B_j}\circ h$, i.e. $f\in S$.

Next,  we show that  if $I_k(B_i,B_j)\subseteq S$, then $I_{k-1}(B_i,B_j)\subseteq S$ for any  $k\leq n$.  Suppose $i\neq j$.
Let $f\in I_{k-1}(B_i,B_j)$, $\dom(f)=\{r_1,\cdots ,r_{k-1}\}$ and $B_i\cap B_j=\{s_1,\cdots ,s_n\}$. There is $h\in I_k(B_i,B_j)$ such that $\dom(h)=\{r_1,\cdots,r_{k-1}, r_k\}$ and  $h(r_i)=s_i$ for each $i\in \{1,\cdots, k-1\}$ and $h(r_k)\not\in \{s_k,\cdots,s_n\}$. Also, there is $g\in S_\infty(B_j)$ such that $g(s_i)=f(r_i)$ for each $i\in \{1,\cdots, k-1\}$. Then we have that $f=g\circ 1_{B_i}\circ h$, and therefore $f\in S$. 

It remains to verify that  $I_{k-1}(B_i,B_i)\subseteq S$, whenever   $I_{k}(B_i,B_i)\subseteq S$. Fix $j\neq i$. Let $f\in I_{k-1}(B_i,B_i)$,  $\dom(f)=\{r_1,\cdots ,r_{k-1}\}$ and $B_i\cap B_j=\{s_1,\cdots ,s_n\}$.  There is $h\in I_k(B_i,B_i)$ such that $\dom(f)\subseteq \dom(h)=\{r_1,\cdots,r_{k-1}, r_k\}$ and $h(r_i)=s_i$ for each $i\in \{1,\cdots, k-1\}$ and $h(r_k)\not\in \{s_k,\cdots,s_n\}\}$. Pick $g\in S_\infty(B_i)$ such that $g(s_i)=f(r_i)$ for each $i\in \{1,\cdots, k-1\}$. Then we have that $f=g\circ 1_{B_j}\circ h$, i.e. $f\in S$.

Finally, we show that $1_\emptyset\in S$. Let $i\neq j$ and choose $g\in I_1(B_i,B_j)$ such that $\im(g)\cap B_i=\emptyset$. Then for any $f\in S_\infty(B_i)$, we have that $f\circ g=1_\emptyset$.

\end{proof}

\begin{prop}
\label{closedsubsemigroup}
Let  $\{B_i\;|\; i\in \N\}$ be a family of infinite subsets of $\N$. Let $n\in \N$ and  $A\subseteq \N$. We set  
$$
S(A)=\displaystyle\bigcup_{i\in A}S_{\infty}(B_i)\cup \bigcup_{i,j\in \N}\bigcup_{k=0}^{n}I_k(B_i,B_j).
$$ 
If $|B_i\cap B_j|\leq n$ for all $i\neq j$,  then $S(A)$ is closed in $I(\mathbb{N})$.
\end{prop}

\begin{proof}
Let $(f_k)_{k\in \mathbb{N}}$ be a sequence in $S$ and $f\in I(\mathbb{N})$ such that $f_k \rightarrow f$. We need to show that $f\in S$. Let  $i_k,j_k\in \N$ be such that $\dom(f_k)\subseteq B_{i_k}$ and $\im(f_k)\subseteq B_{j_k}$. We consider two cases:

(a) Suppose $\dom(f)$ is finite.  Since $f_k \rightarrow f$, there exists $k_0$ such that $\dom(f)\subseteq \dom(f_k)\subseteq B_{i_k}$ for each $k\geq k_0$ (see Proposition \ref{conv}).  Let  $p=|\dom(f)|$, thus $p\leq n$. Analogously, there is $k_1\geq k_0$ such that $\im(f)\subseteq \im(f_k)\subseteq B_{j_k}$ for all $k\geq k_1$. Thus $f\in I_p(B_{i_{k_1}}, B_{j_{k_1}})$.

(b) Suppose $\dom(f)$ is infinite. Since $\dom(f_k)\to \dom(f)$ (see Proposition \ref{conv}), we can assume that $\dom(f_k)$ is infinite for all $k$. Thus, there is $i$ such that $\dom(f_k)=B_i$ for all $k$. Thus $f\in S_\infty(B_i)$.

\end{proof}

From Proposition \ref{closedsubsemigroup} we also get the following. 

\begin{thm}
\label{Polish-sub}
Let $\{B_i\;|\;i\in \N\}$ be a family of infinite subsets of $\N$ such that, for some fixed $n$,   $|B_i\cap B_j|=n$ for all $i\neq j$.   Then, the inverse semigroup generated by $\bigcup_{i\in \N}S_{\infty}(B_i)$ is Polish. \qed 
\end{thm}

Now we present our first example of an inverse semigroup with the Pettis property.

\begin{thm}
\label{Pettis1}
Let $\{B_i\;|\;i\in \N\}$ be a family of infinite subsets of $\N$. Suppose there is $n$ such that  $|B_i\cap B_j|\leq n$ for $i\neq j$. If $n>0$, suppose furthermore that 
 $\{i\in \N:F\subseteq B_i\}$ is finite for all $F\subseteq \N$ of size $n$. Let 
$$
S=\displaystyle\bigcup_{i\in \N}S_{\infty}(B_i)\cup \bigcup_{i,j\in \N}\bigcup_{k=0}^{n}I_k(B_i,B_j).
$$ 
Then $S$ is a Polish inverse semigroup with  the Pettis property.
\end{thm}

\begin{proof} 
From Propositions \ref{2.2} and  \ref{closedsubsemigroup}, $S$ is a Polish inverse subsemigroup of $I(\N)$. We  use Proposition \ref{criterio-1} to show that $S$ has the Pettis property. 
Let $I$ be the collection of isolated points of $S$. By Proposition \ref{closedsubsemigroup} we know that each $S_\infty(B_i)$ is open.  Thus $S\setminus (\bigcup_n S_\infty(B_n)\cup I)$ is closed and countable. Hence, $\bigcup_n S_\infty(B_n)\cup I$ is dense in $S$. So, it remains to verify that 
$xx^{-1}\in I$ for all $x\in I$.

If $n=0$, then there are no isolated points as $1_\emptyset=\lim_i 1_{B_i}$. 
Suppose $n>0$. We  claim that
\[
I=\bigcup_{i,j\in \N}I_n(B_i,B_j).
\]
In fact, let $f\in I_n(B_i,B_j)$ and  $F=\dom(f)$. By hypothesis, there is $k_0$ such that $F\not\subseteq B_k$ for all $k>k_0$. Pick $b_k\in B_k\setminus F$ for all $k\leq k_0$. Let 
$$
V=\bigcap_{a\in F} v(a,f(a))\cap \bigcap_{k\leq k_0} w_1(b_k)\cap S.
$$
We need to show that $V =\{f\}$. Suppose not and let $g\in V$ with $g\neq f$. Then $|\dom(g)|>n$ and it has to be infinite. Thus $F\subseteq \dom(g)=B_k$ for some $k\leq k_0$. But this is impossible, as $b_k\not\in \dom(g)$.  Conversely, if $f\in S_\infty(B_i)$, then $f$ is not isolated, as each $S_\infty(B_i)$ is open in $S$. On the other hand, suppose $f\in I_{k}(B_i, B_j)$  with $k<n$. For each $l\in \N$, pick $n_l\in B_i\setminus \dom(f)$ and $m_l\in B_j\setminus\im(f)$  with $n_l\neq n_{l'}$ for $l'\neq l$. Consider $g_l=f\cup\{(n_l,m_l)\}$. Then $g_l\in I_{k+1}(B_i, B_j)$ and $g_l\to f$. Thus $f$ is not isolated. 

Thus, it is now clear  that 
$xx^{-1}\in I$ for all $x\in I$ and   all hypothesis of Proposition \ref{criterio-1} are satisfied. 

\end{proof}

The following example shows that the Pettis property for Polish inverse semigroups does not imply that a Baire measurable homomorphism is automatically continuous.

\begin{ejemplo}
\label{PettisNotAC}
Suppose $(B_n)_n$ is a pairwise disjoint family of infinite subsets of $\N$. Let $S=\bigcup_n S_\infty(B_n)\cup\{1_\emptyset\}$.  Let $\tau$ be the topology on $S$ as a subspace of $I(\N)$. Then $(S,\tau)$ is a Polish inverse semigroup with the Pettis property (by Theorem \ref{Pettis1}).  We will show that $(S,\tau)$ does not have  automatic continuity with respect to the class of Polish inverse semigroups. Notice that $\bigcup_n S_\infty(B_n)$ is open in $S$ and therefore Polish.  Consider
\[
S=\left (\bigcup_n S_\infty(B_n)\right)\oplus \{1_\emptyset\},
\]
where $\bigcup_nS_\infty(B_n)$ and $\{1_\emptyset\}$ are given the subspace topology. It is well known (and elementary)  that the topological sum of Polish spaces is Polish. Let $\rho$ be this new Polish topology  on $S$.   It is easy to verify that $(S,\rho)$ is a topological inverse semigroup. We only need to show that the following set is $\rho$-open:
\[
V=\{(f,g)\in S\times S:\; f\circ g=1_\emptyset\}.
\]
Suppose $f\in S_\infty(B_i)$, $g\in S_\infty(B_j)$  and $f\circ g=1_\emptyset$. Then, $i\neq i$ and  $(f,g)\in S_\infty(B_i)\times S_\infty(B_j)\subseteq V$. 

Finally, the identity function $f:(S,\tau)\to (S,
\rho)$ is clearly a Borel measurable homomorphism which  is not continuous since $1_\emptyset$ is not $\tau$-isolated. 
\end{ejemplo}

We show next that the previous result does not occur for  a finite partition of $\N$, namely, the generated inverse semigroup has automatic continuity.  Automatic continuity for semigroups is obtained by lifting the same property from a subgroup of the semigroup (\cite{elliott2020,elliott2022,MES}). Thus, we use the following  property of the symmetric group (see also \cite[$\S$ 5]{Rosendal2009}).

\begin{thm}(Kechris-Rosendal, \cite{KechrisRosenthal2007})
\label{AutCon}
$S_\infty(\N)$ has automatic continuity with respect to the class of second countable topological groups.
\end{thm}

\begin{thm}
Let $B_1,\cdots, B_m$ be a collection of infinite subsets of $\N$. Suppose there is $n$ such that $|B_i\cap B_j|\leq n$ for $i\neq j$ and let  $S$ be the semigroup generated by $\bigcup_{i=1}^m S_\infty(B_i)$. Then, $S$  is a Polish semigroup that has the Pettis property. Moreover, if the $B_i$'s are pairwise disjoint, then $S$ has automatic continuity with respect to the class of second countable topological inverse semigroups.
\end{thm}

\begin{proof}
By an argument entirely analogous to that in the proof of Theorems \ref{semigrupogenerado} and Proposition \ref{closedsubsemigroup} we have  that $S$ is closed in $I(\N)$ and 
$$
S=\bigcup_{i=1}^m S_{\infty}(B_i)\cup \bigcup_{i,j=1}^m\bigcup_{k=0}^{n}I_k(B_i,B_j).
$$
Moreover, using the argument in the proof of Theorem \ref{Pettis1}, we have that $S$ has the Pettis property. 

Suppose that $B_1, \cdots, B_m$ are pairwise disjoint.  Then, 
\[
S=\bigcup_{i=1}^m S_\infty(B_i)\cup\{1_\emptyset\}
\]
and each $S_\infty(B_i)$ is a clopen subset of $S$. We will show that $S$ has automatic continuity. Let $T$ be a second countable topological inverse semigroup and $\varphi:S\to T$ be a homomorphism.  Let $\varphi_i:S_\infty(B_i)\to T$ be the restriction of $\varphi_i$  to $S_\infty(B_i)$. Then, the range of $\varphi_i$ is a second countable topological group. Hence by Theorem \ref{AutCon}, $\varphi_i$ is continuous.  Since each $S_\infty(B_i)$ is  clopen, $\varphi$ is also continuous.  
\end{proof}

The next example shows that  Theorem \ref{Pettis1} can not be generalized to any almost disjoint family of sets. It provides an example of an inverse semigroup without the Pettis property. 

\begin{ejemplo}
\label{unpuntoencomun}
Let $\{C_n:\; n\in \N\}$ be a partition of $\N\setminus\{0\}$ and $B_n=C_n\cup\{0\}$. Then $B_n\cap B_m=\{0\}$ for all $n\neq m$. Let $S$ be the inverse semigroup generated by $\bigcup_n S_\infty(B_n)$. Then, 
\[
S=\bigcup_{i\in \N}S_{\infty}(B_i)\cup I_1(\N)\cup \{1_\emptyset\} 
\]
where $I_1(\N)=I_1(\N,\N)$.  From Theorem \ref{Polish-sub} we know that $S$ is a Polish inverse subsemigroup of $I(\N)$. We will show that $S$ does not have neither the Pettis property nor automatic continuity.

We first characterize the isolated points of $S$. Let $u_x$ be the element of $S$ where $\dom(u_x)=\{x\}$ and $u_x(x)=0$. The isolated points of $S$ are
\[
I=\{u_x : x\in \N\setminus\{0\}\}\;\cup \;\{u^{-1}_x : x\in \N\setminus\{0\}\}\;\cup\; \{u_y^{-1}\circ u_x : x,y\in \N\setminus\{0\}\}.
\]
In fact,  given $x,y\in B_n\setminus\{0\}$ with $x\neq y$, $z\in B_m\setminus\{0\}$ and $m,n\in \N$, we have
\[
\{u_x\}=v(x,0)\cap w_1(y)\cap S,
\]
\[
\{u^{-1}_z\circ u_x\}=v(x,z)\cap w_1(0)\cap S.
\]
On the other hand, $u_0$ and $1_\emptyset$ are not isolated:
$u_0=\lim_n1_{B_n}$ and $1_\emptyset=\lim_n 1_{\{n\}}$.

$S$ does not have the Pettis property because  a somewhat trivial reason. Let $A=\{u^{-1}_1\}$, then  $A$ is a closed not meager set as $u^{-1}_1$ is isolated, but  $A^{-1}\circ A=\{u_0\}$ has empty interior. 

Now we show that $S$ does not have automatic continuity. As in Example \ref{PettisNotAC}, it suffices to show that the following topological sum makes $S$ a topological inverse semigroup: 
\[
S=\left (\bigcup_{i\in \N}S_{\infty}(B_i)\cup I_1(\N)\right)\oplus  \{1_\emptyset\} 
\]
Let $\rho$ be the new topology on $S$. 
We only need to show that the following set is $\rho$-open 
\[
V=\{(f,g)\in S\times S:\; f\circ g=1_\emptyset\}.
\]
This can be verified considering several cases. 

\begin{itemize}
    \item[(i)] Suppose $f\in S_\infty(B_i)$ and $f\circ g=1_\emptyset$. Then, $g\in I_1(\N)$ and it is easy to verify that $g$ has to be isolated. Thus $(f,g)\in S_\infty(B_i)\times \{g\}\subseteq V$. Analogously we treat the case  $g\in S_\infty(B_i)$ and $f\circ g=1_\emptyset$.
    
    \item[(ii)] Suppose $f,g\in I_1(\N)$ and $f\circ g=1_\emptyset$.  If $f$ and $g$ are isolated, there is nothing to show. Otherwise, suppose $f=u_0$ and $g\in I_1(\N)$ with $g\neq u_0$. Let $\dom(g)=\{x\}$. Then,  $(f,g)\in w_2(x)\times\{g\}\subseteq V$. The other case is similar. 
\end{itemize}
\end{ejemplo}

\begin{comment}

\begin{thm}
Let $n\in \mathbb{N}$, $I$ be a countable collection and $\{B_i\textit{ }|\textit{ }i\in I\}$ be a family of infinite subsets of the naturals such that for each $i\not=j$, $|B_i\cap B_j|\leq n$. Then the polish inverse semigroup $S=\displaystyle\bigcup_{i\in I}S_{\infty}(B_i)\cup \bigcup_{i\in I}\bigcup_{j\in I}\bigcup_{k=1}^{n}I_k(B_i,B_j)\cup \{1_\emptyset\}$ has Pettis property with $\tau_{pp}$.
\end{thm}

\begin{proof}
Let $A\subseteq S$ with the Baire property and not meagre. Since $I_k(B_i,B_j)$ is a countable set, then is meagre, and therefore $A\cap I_k(B_i,B_j)$ is meagre.\\
Since $A=\displaystyle\bigcup_{i\in I}(A\cap S_{\infty}(B_i))\cup \bigcup_{i\in I}\bigcup_{j\in I}\bigcup_{k=1}^{n}(A\cap I_k(B_i,B_j))\cup \{1_\emptyset\}$, then exist $i_0\in I$ such that $A\cap S_{\infty}(B_{i_0})$ is not meagre.\\

By Corollary \ref{s infinito abierto}, $S_\infty(B_{i_0})$ is open in $S$ and then $A\cap S_\infty(B_{i_0})$ has Pettis property in $S_\infty(B_{i_0})$. Since $S_\infty(B_{i_0})$ is a Baire topological group, and applying the Pettis theorem to $A\cap S_\infty(B_{i_0})$, exist a open $W$ in $S_\infty(B_{i_0})$ such that $1_{B_{i_0}}\in W$, and
$$W\subseteq (A\cap S_\infty(B_{i_0}))^{-1}(A\cap S_\infty(B_{i_0}))\subseteq A^{-1}A.$$

Since $S_\infty(B_{i_0})$ is open in $I(\mathbb{N})$, then $W$ is open in $I(\mathbb{N})$, and this makes that $S$ has Pettis property with $\tau_{pp}$.
\end{proof}

\end{comment}

\bigskip

Now we present a characterization of the inverse semigroup associated to an arbitrary  almost disjoint family. For that we need to introduce a notion.

\begin{defn}
Let $\{B_i\;|\;i\in \N\}$ be an almost disjoint family of infinite subsets of $\N$. Let $m\in \N$ and $i,j\in \N$, an {\em $m$-chain between $B_i$ and $B_j$} is a tuple  $({k_1},...,{k_n})$ such that $|B_{i}\cap B_{k_1}|\geq m$, $|B_{k_1}\cap B_{k_2}|\geq m$, $\cdots$, $|B_{k_{n-1}}\cap B_{k_n}|\geq m$, and $|B_{k_n}\cap B_{j}|\geq m$. We allow  $i=j$, but in this case we require that $i\not=k_1$. 

For $i,j\in \N$, we let
\[
P_{i,j}=\sup \{m\in \N\text{ }|\;\mbox{there exists an $m$-chain between $B_i$ and $B_j$} \}.
\]
\end{defn}

Notice that $P_{i,j}$ could be equal to $\infty$. 

\bigskip 

\begin{thm}
Let $\{B_i\;|\;i\in \N\}$ be an almost disjoint family of infinite subsets of $\N$.  
Let $S$ be the inverse subsemigroup generated by $\bigcup_{i\in \N}S_{\infty}(B_i)$. Then 
\begin{equation}
\label{generado2}
S=\bigcup_{i\in \N}S_{\infty}(B_i)\cup \bigcup_{i,j\in \N}\bigcup_{k=1}^{P_{i,j}}I_k(B_i,B_j)\cup \{1_\emptyset\}.
\end{equation}
\end{thm}

\begin{proof} We start by showing $\subseteq$ in \eqref{generado2}. For that end, it suffices to observe that the right hand side of \eqref{generado2} is an inverse subsemigroup of $I(\N)$. In fact, suppose first that   $f\in I_k(B_i,B_j)$ with $k\leq P_{i,j}$ and $g\in I_l(B_r, B_s)$ with $l\leq P_{r,s}$. Then $f\circ g\in I_q(B_r, B_j)$ for some $q\leq \min\{k,l\}$.  We need to show that  $q\leq P_{r,j}$. In fact, observe that an $l$-chain from $B_r$ to $B_s$ followed by an $k$-chain from $B_i$ to $B_j$ is a $q$-chain from $B_r$ to $B_j$ as $|B_s\cap B_i|\geq q$.  The other cases are similar and left to the reader. 

Conversely, let $m\leq P_{i,j}$ and $f\in I_m(B_i,B_j)$. Let $(k_1, \cdots, k_n)$ be a $m$-chain from $B_i$ to $B_j$.  We will verify the case $n=1$, the rest is completely similar. Suppose $|B_i\cap B_k|\geq m$ and $|B_k\cap B_j|\geq m$. 
Since $|\dom(f)|=m$, pick $g_1\in S_\infty(B_i)$ and $g_2\in S_\infty(B_k)$ such that
\[
\begin{array}{l}
g_1(\dom(f))  \subseteq  B_i\cap B_k,\\ \\
g_2(g_1(\dom(f))) \subseteq  B_k\cap B_j,\\ \\
g_2((B_k\cap B_j)\setminus g_1(\dom(f))) \cap  B_k\cap B_j=\emptyset. \\
{}
 \end{array}
 \]
 Finally, pick $g_3\in S_\infty(B_j)$ such that  $g_3(g_2(g_1(x)))=f(x)$ for all $x\in \dom(f)$. The last condition above guarantees that $\dom (g_3\circ g_2\circ g_1)=\dom (f)$. Therefore $f=g_3\circ g_2\circ g_1$ and $f\in S$. 

\end{proof}

\section{More subsemigroups of $I(\N)$}

We present two general constructions of  subsemigroups of $I(\N)$. The idea is quite natural: we impose some restrictions on the domain and range of the functions.    We construct several Polish inverse subsemigroup of $I(\N)$ without  Pettis property. In particular, $I(\N)$ does not have this property.

For each collection $\mathcal{C}$ of subsets of $\N$ we put
\[
\mathbb{S}(\mathcal{C}) = \{h\in I(\N):\; \dom(h)\in \mathcal{C}\;\textit{and}\; \im(h) \in \mathcal{C}\}
\]
and
\[
\mathbb{S}^+(\mathcal{C}) = \{h\in I(\N): (\N\setminus\dom(h))\in \mathcal{C}\;\textit{and}\; (\N\setminus\im(h))\in \mathcal{C}\}. 
\]
Notice that $\mathbb{S}^+(\mathcal{P}(\N))=I(\N)$.

A collection $\ideal$ of subsets of $\N$ is an  {\em ideal}, if $\ideal$ hereditary (i.e. if $A\subseteq B\in \ideal$, then $A\in \ideal$) and  closed under finite unions. An ideal is {\em proper} if $\N$ does not belong to the ideal. We will use  ideals which are $F_\sigma$ as subsets of $\cantor$. There are plenty of such ideals (see, for instance, \cite{uzcasurvey}).  Suppose $\ideal$ is a proper  ideal  of subsets of $\N$ and let $\ideal^+$ be $\mathcal{P}(\N)\setminus \ideal$. Then 
\[
\mathbb{S}(\ideal)\subseteq \mathbb{S}^+(\ideal^+).
\]

\begin{prop}
\label{S(C)}
Let $\mathcal{C}$ be a collection of subsets of $\N$. 
\begin{itemize}
\item[(i)] If $\mathcal{C}$ is hereditary, then $\mathbb{S}(\mathcal{C})$ is an inverse subsemigroup of $I(\N)$. 

\item[(ii)] If $\mathbb{S}(\mathcal{C})$ is an inverse subsemigroup, then $\mathcal{C}$ is closed under intersections.

\item[(iii)] If $\mathcal{C}$ contains all finite sets, then $\mathbb{S}(\mathcal{C})$ is dense in $I(\N)$.

\item[(iv)] If $\mathcal{C}$ is a closed subset of $\cantor$, then $\mathbb{S}(\mathcal{C})$ is closed in $I(\N)$.
\item[(v)]If $\mathcal{C}$ is upward closed, then $\mathbb{S}^+(\mathcal{C})$ is an inverse subsemigroup of $I(\N)$. If in addition, $\mathcal{C}$ is $G_\delta$, then so is $\mathbb{S}^+(\mathcal{C})$ and therefore it is a Polish inverse semigroup. 

\end{itemize}
\end{prop}

\begin{proof}

\begin{itemize}
\item[(i)] Suppose $f,g\in \mathbb{S}(\mathcal{C})$. Notice that $\dom(f\circ g)\subseteq \dom(g)$ and $\im(f\circ g)\subseteq \im(f)$. Since $\mathcal{C}$ is hereditary, $f\circ g\in \mathbb{S}(\mathcal{C})$. 

\item[(ii)] Let $A,B\in \mathcal{C}$. Since $1_{A}\circ 1_{B}=1_{A\cap B}\in \mathbb{S}(\mathcal{C}),$ then $A\cap B\in \mathcal{C}.$ 

\item[(iii)] 
Let $\{x_i: \;1\leq i\leq n\}$, $\{y_i: \;1\leq i\leq n\}$, $\{u_i: \;1\leq i\leq m\}$ and $\{z_i: \;1\leq i\leq l\}$ be finite subsets of $\N$ such that $\{x_i: \;1\leq i\leq n\}\cap \{u_i: \;1\leq i\leq m\}=\emptyset$, $\{y_i: \;1\leq i\leq n\}\cap \{z_i: \;1\leq i\leq l\}=\emptyset$. Consider the following basic open set of $I(\N)$:
\[
V=\bigcap_{i=1}^n v(x_i,y_i)\cap \bigcap_{j=1}^m w_1(u_j)\cap \bigcap_{k=1}^l w_2(z_k).
\]
Let $f=\{(x_i,y_i):\; 1\leq i\leq n\}$. As $\mathcal{C}$ contains all finite sets, $f\in V\cap \mathbb{S}(\mathcal{C})$.

\item[(iv)] It follows from the continuity of the functions $\dom, \im:I(\N)\to \cantor$. 
\item[(v)] Similar to (i) and (iv). It is a well known classical theorem that every $G_\delta$ subset of a Polish space is also Polish.
\end{itemize}
\end{proof}
\bigskip

In the previous proposition, the converse of (ii) is not true. Consider $\mathcal{C}=\{\{0,\cdots, n\}:\; n\in \N\}$ and the function $f:\{0,1\}\to \{0,1\}$ given by  $f(0)=1$ and $f(1)=0$.  Notice that  $f\circ 1_{\{0\}}\notin \mathbb{S}(\mathcal{C})$.

\bigskip

Let $[\N]^{\leq n}$ be the collection of all subsets of $\N$ with at most $n$ elements. 
The inverse semigroup  $\mathbb{S}([\N]^{\leq n})$ has been studied in the literature  (see, for instance, \cite{Gutik2009b,Gutik-etal-2009}).  It is easy to verify that  $[\N]^{\leq n}$ is a closed subset of $\cantor$, thus from Proposition \ref{S(C)} we get 

\begin{ejemplo}
$\mathbb{S}([\N]^{\leq n})$ is a Polish inverse subsemigroup of $I(\N)$. 
\end{ejemplo}

\begin{ejemplo}
Consider the Schreier family
\[
\mathcal{C}=\{F\subseteq \N:\; |F|\leq \min(F)+1\}.
\]
$\mathcal{C}$ is a closed subset of $\cantor$ homemorphic to the ordinal  $\omega^\omega+1$ (with the order topology). 
Hence $\mathbb{S}(\mathcal{C})$ is a Polish inverse subsemigroup of $I(\N)$. It can be shown that it has Cantor-Bendixon rank equal to $\omega$.
\end{ejemplo}

Our next result provides examples of Polish inverse semigroups without the Pettis property. Which moreover do not have isolated points (in contrast with Example \ref{unpuntoencomun}).

\begin{thm}
\label{nonPP}
Let $\ideal$ be a proper $F_\sigma$ ideal of subsets of $\N$ containing all finite sets. Then $\mathbb{S}^+(\ideal^+)$ is a Polish inverse subsemigroup of $I(\N)$ without the Pettis property.
\end{thm}

\begin{proof} 
Let $S=\mathbb{S}^+(\mathcal{I}^+)$. Since $\ideal$ is 
hereditary, $\mathcal{I}^+$ is upward closed and,  by Proposition \ref{S(C)}, $S$ is an inverse semigroup.  Since $\mathcal{I}^+$ is $G_\delta$ as a subset of $\cantor$, $S$ is a $G_\delta$ subset of $I(\N)$ and thus it is a Polish inverse semigroup.  

Now we show that $S$ does not have the Pettis property. Let  $A\subseteq \N$ with $A, A^c\not\in \ideal$ and $\mathcal{J}$ be the  ideal generated by $\ideal\cup\{A\}$. Then $\mathcal{J}$ is  a proper  $F_\sigma$ ideal  properly extending $\ideal$. Let  $L=\mathbb{S}^+(\mathcal{J}^+)$. Since $\mathcal{J}^+\subseteq \ideal^+$, $L\subseteq S$.   We will show that $L$ is a witness  that $S$ does not have the Pettis property. 

First of all,    $L$ is an inverse semigroup and thus $L=L^{-1}$ and $L\circ L=L$. Since $\mathcal{J}^+$ is $G_\delta$ as a subset of $\cantor$, so is $L$ as a  subset of $S$. Moreover, $\mathbb{S}(\mathcal{J})\subseteq L$ and every finite set belongs to $\mathcal{J}$, thus $L$ is dense in $I(\N)$ (by Proposition \ref{S(C)}). Hence $L$ is non meager in $S$.  

It remains to verify  that $L$ has empty interior in $S$.
We need to show that $S\setminus L$ is dense in $S$. 
Let $\{x_i: \;1\leq i\leq n\}$, $\{y_i: \;1\leq i\leq n\}$, $\{u_i: \;1\leq i\leq m\}$ and $\{z_i: \;1\leq i\leq l\}$  be finite subsets of $\N$ such that $\{x_i: \;1\leq i\leq n\}\cap \{u_i: \;1\leq i\leq m\}=\emptyset$, $\{y_i: \;1\leq i\leq n\}\cap \{z_i: \;1\leq i\leq l\}=\emptyset$. Let $V$ be a non empty basic open set of $I(\N)$:
\[
V=\bigcap_{i=1}^n v(x_i,y_i)\cap \bigcap_{j=1}^m w_1(u_j)\cap \bigcap_{k=1}^l w_2(z_k).
\]
Let $X=\{x_i:\;1\leq i\leq n\}$, $Y=\{y_i:\;1\leq i\leq n\}$, $U=\{u_j:1\leq j\leq m\}$ and $Z=\{z_k:1\leq k\leq l\}$. 
Let $B=A^c\setminus (X\cup Y\cup U\cup Z)$ and  $f=1_B\cup\{(x_i,y_i):1\leq i\leq n\}$.
Clearly $f\in V$. On the other hand, 
\[
\dom(f)^c=(B\cup X)^c=(A\cup  X\cup Y\cup U\cup Z)\cap X^c.
\]
Since $A=(A\cap X)\cup (A\cap X^c)$, $X\in \ideal$ and $A\not\in \ideal$, $\dom(f)^c\in \mathcal{J}\setminus\ideal$. Analogously, $\im(f)^c\in \mathcal{J}\setminus\ideal$. Thus $f\in V\cap (S\setminus L)$. 

\end{proof}

For $\ideal=\emptyset$, $\ideal^+=\mathcal{P}(\N)$ and $I(\N)=\mathbb{S}^+(\ideal^+)$.   
The reader can verify that for $\ideal=\emptyset$ in the proof of Proposition \ref{nonPP}  we still can conclude that $f\in V\cap (S\setminus L)$. Thus, we get the following. 

\begin{cor}
$I(\N)$ does not have the Pettis property.
\end{cor}

\bibliographystyle{plain}

\begin{thebibliography}{9}

\bibitem{Carruthetal}
Carruth, J., Hildebrant, H and Koch, R.
\newblock {\em The theory of topological semigroups}, volume~75 of {\em
  Monographs and Textbooks in Pure and Applied Mathematics}.
\newblock Marcel Dekker, Inc., New York, 1983.


\bibitem{elliott2020}
Elliott, L., Jonušas, J.,  Mesyan, Z.,  Mitchell, J. D., Morayne, M. and Péresse, Y.
\newblock Automatic continuity, unique {P}olish topologies, and {Z}ariski
  topologies (Part I, monoids).
\newblock {arxiv.org/abs/1912.07029v3}
\newblock 2020.

\bibitem{elliott2022}
L.~Elliott, J.~Jonu\v{s}as, J.~D. Mitchell, Y.~P\'eresse, and M.~Pinsker.
\newblock Polish topologies on endomorphism monoids of relational structures.
\newblock {arXiv:2203.11577}, 2022.

\bibitem{Gutik2009b}
O.~V. Gutik and A.~R. Reiter.
\newblock Symmetric inverse topological semigroups of finite rank {$\leq n$}.
\newblock {\em Mat. Metodi Fiz.-Mekh. Polya}, 52(3):7--14, 2009.

\bibitem{Gutik-etal-2009}
Oleg Gutik, Jimmie Lawson, and Du\v{s}an Repov\v{s}.
\newblock Semigroup closures of finite rank symmetric inverse semigroups.
\newblock {\em Semigroup Forum}, 78(2):326--336, 2009.


\bibitem{Howie} 
Howie, J. 
\newblock {Fundamentals of semigroup theory},  {London  Mathematical Society Monographs. New Series} 12.
\newblock Oxford University Press, New York, Second
  Edition, 2003.


\bibitem{kechris1995} Kechris, A.  Classical Descriptive Set Theory. Graduate Texts in Mathematics 156. Springer-Verlag, New York, 1995.

\bibitem{KechrisRosenthal2007}
A.~Kechris and C.~Rosendal.
\newblock Turbulence, amalgamation, and generic automorphisms of homogeneous structures.
\newblock {\em Proc. Lond. Math. Soc. (3)}, 94(2):302--350, 2007.
\bibitem{MES} Mesyan, Z.,  Mitchell, J. and P\'eresse, Y. Topological transformation monoids. arXiv:1809.04590, 2018.

\bibitem{PerezUzca2022}
P\'erez, J. and Uzc\'ategui C. 
\newblock Topologies on the symmetric inverse semigroup.
\newblock {\em Semigroup Forum}, 2022.


\bibitem{Rosendal2009}
Rosendal, C.
\newblock Automatic continuity of group homomorphisms.
\newblock {\em Bull. Symbolic Logic}, 15(2):184--214, 2009.

\bibitem{RosendalSolecki2007}
Rosendal, C. and Solecki, S.
\newblock Automatic continuity of homomorphisms and fixed points on metric
  compacta.
\newblock {\em Israel J. Math.}, 162:349--371, 2007.

\bibitem{Sabok2019}
Sabok, M.
\newblock Automatic continuity for isometry groups.
\newblock {\em J. Inst. Math. Jussieu}, 18(3):561--590, 2019.

\bibitem{uzcasurvey}
Uzc\'{a}tegui-Aylwin, C.
\newblock Ideals on countable sets: a survey with questions.
\newblock {\em Rev. Integr.temas mat.}, 37(1):167--198, 2019.


\end{thebibliography}

\end{document}